\documentclass[11pt,leqno]{amsart}
\usepackage{textcomp} 
\usepackage{indentfirst} 
\usepackage{amsmath}
\usepackage{amsthm}
\usepackage{amsfonts}
\usepackage{amssymb}
\usepackage{eucal}
\usepackage{mathrsfs}
\usepackage{stmaryrd}
\usepackage{bbm}
\usepackage{graphicx}
\usepackage[cp1250]{inputenc}

\def\C{\mathbb{C}}%
\def\Z{\mathbb{Z}}%
\def\T{\mathbb{T}}%
\def\D{\mathbb{D}}%
\def\1{\bf 1}%
\DeclareMathOperator{\id}{id}

\newtheorem{thm}{Theorem}[section]

\newtheorem{cor}[thm]{Corollary}
\newtheorem{lem}[thm]{Lemma}
\theoremstyle{proposition}
\newtheorem{prop}[thm]{Proposition}
\theoremstyle{definition}

\theoremstyle{remark}

\newtheorem{rem}[thm]{Remark}

\begin{document}
\title{Toeplitz operators on backward shift invariant subspaces of $H^p$}
\author{Maria Nowak and Andrzej So{\l}tysiak}

\address{Maria  Nowak,  \newline Institute of Mathematics,
\newline Maria Curie-Sk{\l}odowska University, \newline pl. M.
Curie-Sk{\l}odowskiej 1, \newline 20-031 Lublin, Poland}
\email{mt.nowak@poczta.umcs.lublin.pl}
\address{
Andrzej So{\l}tysiak  \newline Faculty of Mathematics and Computer
Science,
\newline Adam Mickiewicz University, ul. Umultowska 87, \newline 61-614 Pozna\'n, Poland}
\email{asoltys@amu.edu.pl}


\begin{abstract}

We extend  results on compressed Toeplitz operators on  the backward
shift invariant subspaces of $H^2 $ to the context of the spaces
$H^p$, $1<p<\infty.$
\end{abstract}

\maketitle

\section{Introduction}
Let $\D$ denote the unit disc  in the complex plane and let $\T$ be
its boundary. For $1\leqslant p< \infty$  let $H^p$ be the 
Hardy space of functions in $L^p({\T})$ with vanishing negative
Fourier coefficients. $H^p$ consists of the boundary values of
functions holomorphic in $\D$ and satisfying
\[
\sup\limits_{0<r<1}\frac{1}{2\pi}\int_0^{2\pi} \vert
f(re^{i\theta})\vert ^p\,d\theta<\infty.
\]
Let $H^\infty$ denote the Banach space of all bounded holomorphic
functions in $\D$ equipped with the usual supremum norm.
 We will write $H_0^p$ for the subspace of
$H^p$ consisting of functions vanishing at zero and $\overline{H^p}$
for the subspace of the complex conjugates of functions in $H^p$.

Let $S^\ast$ denote the backward shift operator on $H^p$ defined in
the unit disc $\D$ by
\[
(S^\ast f)(z)=\frac{f(z)-f(0)}z.
\]
A closed subspace $\mathcal{M}$ of $H^p$ is called backward shift
invariant,  or $ S^\ast$- invariant, if $f\in{\mathcal M}$ implies
$S^\ast f\in{\mathcal M}$. It is well known that  for  $1<p<\infty$,
all $S^\ast$-invariant subspaces are of the form
\[
\mathcal{M}=K_I^p=H^p\cap I\overline{H_0^p}=(IH^q)^\bot
\]
for some inner function $I$, where $(IH^q)^\bot$ is the annihilator
of the space $IH^q$ and $\frac 1p+\frac 1q=1$ (see \cite[pp.
81--83]{CR}).

For $p=2$ the space $K_I=K_I^2$ is the orthogonal complement of the
shift invariant subspace $IH^2$ and is called a model space. The
model spaces play an important role in the model theory for Hilbert
space contractions (\cite{N3},\cite{SzNB}). In this paper we deal
mainly with $K^p_I$, $1<p<\infty$, and Toeplitz operators on these
spaces. Recently these spaces have been studied e.g. in \cite{Dy1},
\cite {Dy2}, \cite {HS}. The theory of model spaces is based on
Hilbert space methods that are not always easily transferred to the
context of general $H^p$ spaces. Often, however, results in $H^p$
spaces are suggested  by those  in $H^2$.

In Section 3 we prove that the Toeplitz operator $T_{\bar{a}}$ with
the bounded co-analytic symbol $\bar{a}$ restricted to $K^p_I$ is
invertible if and only if the functions $a$ and $I$ form a corona
pair.  This extends the result  obtained by
 P.\,A. Fuhrmann  for $p=2$ (\cite {F1}, \cite {F2} , see also \cite[pp. 17--19] {P}) to  $1<p<\infty.$

In Section 4 we refer to  the commutant of the compressed shift, or,
equivalently,  the commutant  of the restricted backward shift,
investigated by Sarason \cite{S1} in 1967. The description of the
commutant of $S^\ast$ restricted to $K^p_I$ was
  conjectured  in  \cite[p.111]{CR}. Here we give  its proof.

Our  proofs are based  upon  ideas similar to those for the case
 $p=2$.  However,  we   apply neither the functional calculus, nor the
commutant lifting theorem. The main tools we use in our reasoning
are the bounded projection $P_I$ of $H^p$ onto $K^p_I$ and the
representation of linear functionals on $H^p$.

\section{Backward shift invariant spaces}
Let us recall that a closed subspace $\mathcal M$ of a Banach space
$X$ is complemented if there exists a closed subspace $\mathcal{N}$
of $X$ such that $X=\mathcal{M}\oplus\mathcal{N}$, i.e.
$X=\mathcal{M}+\mathcal{N}$ and $\mathcal{M}\cap \mathcal{N}=\{0\}$.

For $1<p<\infty$, the M. Riesz theorem implies that $H^p$ is a
complemented subspace of $L^p({\T})$ and
$H^p\oplus\overline{H^p_0}=L^p({\T})$. Moreover, the Riesz
projection $P_+$ defined by
\[
(P_{+}f)(z)=\frac{1}{2\pi}\int_0^{2\pi}\frac{f(e^{i\theta})}{1-e^{-i\theta}z}\,d\theta
\]
maps $L^p({\T})$ onto $H^p$. We also set $P_{-}=\id-P_{+}$, where $\id$ denotes the identity operator on $L^p(\T)$. 

Now we  show the following

\begin{prop}
For $1<p<\infty$ and any inner function $I$, the space $K_I^p$ is a
complemented subspace of $H^p$.
\end{prop}

\begin{proof}
In view of the known characterization of complemented subspaces (see
\cite{R}, p. 126) it is enough to show that there exists a
continuous projection $P_I$ from $H^p$ onto $K^p_I$. To this end
note that for $f\in H^p$,
\[
\bar{I}f= P_{+}(\bar{I}f)+P_{-}(\bar{I}f).
\]
Since
\begin{equation}
f= I\bar{I}f= IP_{+}(\bar{I}f)+IP_{-}(\bar{I}f)\label{jeden}
\end{equation}
and
\[
P_{+}(\bar{I}f)\cap P_{-}(\bar{I}f)=\{0\},
\]
we get
\begin{equation}
P_If=IP_{-}(\bar{I}f). \label{projection}
\end{equation}
\end{proof}
\begin{rem}
It was proved in \cite[Thm. 2.2]{St} that if $I$ is an inner
function, then the operator $T_{\bar{I}}$ is bounded on $H^1$ if and
only if $I$ is a finite Blaschke product. Consequently, the
projection $P_I$ is bounded only for such inner functions $I$. Hence
Proposition 2.1 is not generally  true for $p=1$.

 It is well known that
for $1<p<\infty$ the dual space $(H^p)^\ast$ can be identified with
$H^q$, $\frac 1p+\frac 1q= 1$, via the pairing
\begin{equation}
\langle f, g\rangle=\frac
1{2\pi}\int_0^{2\pi}f(e^{i\theta})\overline{g(e^{i\theta})}\,d\theta\qquad
(f\in H^p,\  g\in H^q).\label{pairing}
\end{equation}
Clearly, the definition of  $\langle f, g\rangle$ can be extended
for $f\in L^p({\T})$ and $g\in L^q ({\T})$. In the sequel we often
use this symbol  for such functions.

\end{rem}
\begin{rem}
It is worth-while to notice that for $1<p<\infty$ the following
equality holds true
\[
H^p=IH^p\oplus K_I^p.
\]
In view of (\ref{jeden}) it is enough to observe that if  $f\in
H^p$, then $IP_{-}(\bar{I}f)\in (IH^q)^{\bot}$. Indeed, for any
$h\in H^q$ we have
\[
\langle Ih, IP_{-}(\bar{I}f)\rangle=\langle h,
P_{-}(\bar{I}f)\rangle=0.
\]

\end{rem}
\section{Restricted Toeplitz operators}
For $\varphi\in L^\infty({\T})$  the Toeplitz operator $T_\varphi$
on $H^p$, $1< p<\infty$,  is defined by
\[
T_\varphi(f)=P_{+}(\varphi f).
\]

It is easy to check that if $\varphi\in H^\infty$ is outer, then
both the operators $T_\varphi$ and $T_{\bar{\varphi}}$ are
injective. Moreover, we have
\begin{prop}
For $\varphi\in H^\infty$, the operator $T_{\bar{\varphi}}$ is
surjective if and only if $T_\varphi$ is left invertible.
\end{prop}

\begin{proof}
If $\varphi=I\varphi_0$ with $I$ the inner factor and $\varphi_0$
the outer factor  of $\varphi$, then $\ker T_{\bar{\varphi}}=K^p_I$.
Since $K_I^p$ is a complemented subspace of $H^p$, surjectivity of
$T_{\bar{\varphi}}$ is equivalent to right invertibility of
$T_{\bar{\varphi}}$ (see \cite[p. 92, Thm. 16]{M}).
\end{proof}

For any $a\in H^\infty$ we have $T_{\bar{a}}(K^p_I)\subset K^p_I$.
To see this inclusion it is enough to observe that for $f\in K^p_I$
and $g\in H^q$,
\[
\langle T_{\bar{a}}f, Ig\rangle=\langle \bar{a}f, Ig\rangle =\langle
f,aIg \rangle=0,
\]
since $K^p_I=(IH^q)^\bot$.

Two functions $f,g\in H^\infty$ are said to form a corona pair if
\[
\inf\{\vert f(z)\vert+\vert g(z)\vert\colon\,z\in\D\}=\delta>0.
\]
The next theorem is a generalization of a result due to Fuhrmann for
$p=2$ (\cite{F1}, \cite[Thm. 2.7]{P}).

\begin{thm}
Assume that $1<p<\infty$,  $a\in H^\infty$, and $I$ is an inner
function. Then the restriction of the operator $T_{\bar{a}}$ to
$K^p_I$ is invertible  if and only if the functions $a$ and $I$ form
a corona pair.
\end{thm}
\begin{proof}
Suppose first that functions $a$ and $I$  form a corona pair. We
show that the operator  $T_{\bar a}$ is invertible on $K^p_I$. By
the corona theorem there exist functions $u,v\in H^{\infty}$ such
that
\[
a(z)u(z)+I(z)v(z)=1, \quad z\in\D.
\]

Hence for $f\in H^p$
\begin{equation}
f=T_1f= T_{\bar{a}\,\bar{u}+\bar{I}\bar{v}}f= T_{\bar{
a}}T_{\bar{u}}f+ T_{\bar{I}}T_{\bar{v}}f.\label{corona}
\end{equation}
Note that if $f\in K_I^p$, then by (\ref{projection})
\[
T_{\bar{I}}T_{\bar{v}}f=T_{\bar{v}}T_{\bar{I}}f= T_{\bar{
v}}T_{\bar{I}}(IP_{-}(\bar{I}f ))=0.
\]
In view of (\ref{corona}) we get
\[
f=T_{\bar u}T_{\bar a}f= T_{\bar a}T_{\bar u}f,
\]
which  proves invertibility of $T_{\bar a}$ on $K^p_I$.

Assume now that the functions $a$ and $I$ do not form a corona pair.
Then there exists a sequence
 $\{z_n\}$ of points from  $\mathbb D$ such that
\[
\lim_{n\to\infty}(|I(z_n)|+|a(z_n)|)=0.
\]
 Put
\[
b_n(z)=\frac {z-z_n}{1-\overline {z}_n z}
\]
and define functions $I_n$ and $a_n$ by
\[
b_nI_n=I-I(z_n),\qquad b_na_n=a-a(z_n).
\]
Finally, let
\[
k_n(z) =\frac{(1-|z_n|^2)^{1/q}}{1-\bar{z}_n z},\quad \text{where} \
\ \frac{1}{p}+\frac{1}{q}=1.
\]
 We will also need the following asymptotic estimates for the  integral means (see, e.g. \cite[Thm.
 1.12]{Z}). For $z\in\mathbb D$,

\begin{equation}
\left(\int_0^{2\pi}\frac{d\theta}{|1-
ze^{-i\theta}|^p}\right)^{1/p}\sim\frac{1}{ (1-|z|^2)^{1/q}}\quad
\text{as } |z|\rightarrow 1^{-}.\label{norma}
\end{equation}
Let $f_n$ be defined by
\[
f_n=P_{I}I_nk_n.
\]
First we  show that $\|f_n\|_{H^p}\geqslant \gamma>0$ for  $n$ large 
enough and some constant  $\gamma>0$. We have
\begin{eqnarray*} \|f_n\|_{H^p}
&=&
\|P_{-}(\bar I I_nk_n)\|_{L^p}=\|P_{-}((\bar  I-\bar I(z_n))I_nk_n)\|_{L^p}\\
&=& \|P_{-}(\bar{I}_n \overline{b}_n)I_nk_n\|_{L^p}=\|P_{-}(|I_n|^2
\overline{b}_n k_n\|_{L^p}\\
& \geqslant& \|P_{-}( (\bar{b}_n k_n))\|_{L^p}- \|P_{-}((|I_n|^2-1)
\bar{b}_n k_n)\|_{L^p}.
\end{eqnarray*}
Since for $|z|=1$,
\[
P_{-}(\bar{b}_nk_n)(z) = (1-|z_n|^2)^{1/q}\frac
1{z-z_n}=(1-|z_n|^2)^{1/q}\frac{\bar{z}}{1-z_n\bar{
z}},\label{kernel}
\]
estimates (\ref{norma}) imply that there exists $c>0$ such that
$\|P_{-}( \bar{b}_n k_n)\|_{L^p}\geqslant c$.

Furthermore,  $ \|P_{-}((|I_n|^2-1) \bar{b}_nk_n)\|_{L^p}\to 0$ as
$n\to\infty$ (see \cite[p. 18]{P}).

We now  claim that $\|T_{\bar{a}}f_n\|_{H^p}\rightarrow 0$ as
$n\rightarrow\infty$.

It follows from the above  that $T_{\bar{a}}(K^p_I)\subset K^p_I$.
Moreover, since $IH^q$ is the annihilator of $K_I^p$, we get
\begin{eqnarray*}
\|T_{\bar a}f_n\|_{H^p}&=& \sup_{\substack{g\in L^q,\\
\|g\|=1}}|\langle T_{\bar a}f_n, g\rangle|= \sup_{\substack{g\in
L^q,\\ \|g\|=1}}|\langle
T_{\bar a}f_n, P_{+}g\rangle|\\
&=&  \sup_{\substack{g\in L^q,\\
\|g\|=1}}|\langle T_{\bar a}f_n, IP_{-}(\bar I P_{+}g)\rangle |= \sup_{\substack{g\in L^q,\\
\|g\|=1}}|\langle P_{+}(\bar af_n), IP_{-}(\bar I P_{+}g)\rangle |\\
&=& \sup_{\substack{g\in L^q,\\
\|g\|=1}}|\langle \bar af_n, IP_{-}(\bar I P_{+}g)\rangle |= \sup_{\substack{g\in L^q,\\
\|g\|=1}}|\langle \bar aP_{-}(k_nI_n\bar I), P_{-}(\bar I
P_{+}g)\rangle |.
\end{eqnarray*}
Since $b_nI_n=I-I(z_n)$, we have $I_n= \overline{b}_n(I-I(z_n))$.
Therefore
\[
\|T_{\bar{a}}f_n\|_{H^p}= \sup_{\substack{g\in L^q,\\
\|g\|=1}}|\langle \bar{a}(P_{-}(k_n\bar{b}_n)-
I(z_n)P_{-}(\bar{I}k_n\bar{b}_n)), P_{-}(\bar{I} P_{+}g)\rangle |.
\]
It is clear that
\[
\sup_{\substack{g\in L^q,\\
\|g\|=1}}|\langle \bar{a} I(z_n)P_{-}(\bar{I}k_n\bar{b}_n),
P_{-}(\bar I P_{+}g\rangle |\rightarrow 0 \quad \text{as} \ \
n\to\infty.
\]
Moreover,
\begin{eqnarray*}
\sup_{\substack{g\in L^q,\\
\|g\|=1}}|\langle \bar{a} P_{-}(k_n\bar{b}_n) , P_{-}(\bar{I}
P_{+}g)\rangle |&=&\sup_{\substack{g\in L^q,\\
\|g\|=1}} (1-|z_n|)^{1/q}\left|\left\langle \frac {\bar{a}\bar
{z}}{1-z_n\bar{z}}
, P_{-}(\bar{I} P_{+}g)\right\rangle \right|\\
&=& \sup_{\substack{g\in L^q,\\
\|g\|=1}} (1-|z_n|)^{1/q}\left|\left\langle \frac {\bar{a}\bar
{z}}{1-z_n\bar z} , \bar{g}_0\right\rangle\right |,
\end{eqnarray*}
where $g_0=P_{-}(\bar{I} P_{+}g)\in H^q_0$.

Since
\[
 \left\langle \frac {\bar{a}\bar{z}}{1-z_n\bar{z}} ,
\bar{g}_0\right\rangle =\left\langle \bar{a} \bar{z}g_0, \frac {
1}{1-z\bar{z}_n}\right\rangle= P_{+}(\bar{a}\bar{z}g_0)(z_n)
\]
and
\[
|P_{+}(\bar{a} \bar{z}g_0)(z_n)|
=o\left((1-|z_n|^2)^{-1/q}\right)\quad \text{as} \ \  |z_n|\to 1,
\]
our claim follows.
\end{proof}

\section{The commutant of the restricted backward shift}

Let us recall that the commutant $\{S^\ast\}'$ of the backward shift
consists of all bounded operators $A$ on $H^p$, $1<p<\infty$,
commuting with $S^\ast$, i.e.
\[
\{S^\ast\}'=\{A\in\mathcal{L}(H^p)\colon\,AS^\ast=S^\ast A\}.
\]
It is well-known that (see e.g. \cite[pp. 109--110]{CR})
\[
\{S^\ast\}'=\{T_{\bar{\varphi}}\colon\,\varphi\in H^\infty\}.
\]

Here we describe $\{S^\ast | K^p_I \}'$ the commutant  of the
restricted backward shift operator $S^\ast$ to the subspace $K^p_I$,
$1<p<\infty$. For $p=2$ the commutant of this operator was characterized
by Sarason \cite{S1}. The result says
\[
\{S^\ast| K^2_I\}'=\{T_{\bar{\varphi}}\,| K^2_I\colon\,\varphi\in
H^\infty\}.
\]

We will extend this result to the  subspaces $K^p_I$, $1<p<\infty$,
using an approach analogous to that suggested by N. K. Nikolskii for
$p=2$ \cite[pp. 179--182]{N1} (see also \cite[pp.13--15 ]{P}). To
this end, we first define Hankel operators on $H^p$, $1<p<\infty$.

 For $\psi\in L^{p}(\mathbb T)$, $1\leqslant p<\infty$,  we
define the Hankel operator $H_{\psi}$ on the dense subset of $H^p$
(e.g. $H^{\infty}$ or analytic polynomials) by
\[
H_{\psi}f= \psi f-P_{+}(\psi f)=P_{-}(\psi f).
\]

It is easy to see that $H_{\psi}$ is bounded on $H^p$ if the
function $\psi\in L^\infty(\T)$.

For $n\in \Z$ let  $\chi_n(z)= z^n$, $z\in {\T}$. The functions
$\chi_n$, $n\in \Z$, form a Schauder basis for the space $L^p(\T)$.
Let $\mathcal{S}$ denote the bilateral shift on $L^p(\mathbb{T})$,
i.e. ${\mathcal S}\chi_n=\chi_{n+1}$ for $n\in\Z$ and $S={\mathcal
S}|H^p$ be the unilateral shift on $H^p$.

The next theorem contains the known  characterizations of Hankel
operators on $H^p$, $1<p<\infty$. We include its proof  for the
convenience of the reader. We note that a version of this theorem
can be found in \cite[pp. 54--55, Thm. 2.11]{BS}.

\begin{thm} The following statements are equivalent\/{\rm:}
\begin{enumerate}
\item[(i)]
$A\colon\,H^p\to \overline{H^p_0}$ is a bounded linear operator such
that $\langle A\chi_k, \bar{\chi}_j\rangle =a_{j+k}$ for $j\geqslant
1, k\geqslant 0$\/{\rm;}
\item [(ii)]
there exists $\psi\in L^{\infty}(\mathbb{T})$ such that
$A=H_{\psi}$\/{\rm;}
\item [(iii)] $A$ is a bounded linear operator from $H^p$ into $\overline{H^p_0}$ such that
$P_{-}(\mathcal{S}A)=AS$.
\end{enumerate}
\end{thm}

\begin{proof}
(i)\ $\Rightarrow$\ (ii): By assumption the equality
\begin{equation}
\langle Ag,\bar{h}\rangle=\langle A\chi_0,\bar{g}\bar{h}\rangle.
\label{chi}
\end{equation}
holds for any polynomials  $g$ and $h$, $h(0)=0$. Since polynomials
are dense in $H^p$, the above equality is true for all  $g\in H^p$
and $h\in{H^q_0}$. Consequently,
\[
\vert\langle Ag,\bar{h}\rangle\vert\leqslant\Vert Ag\Vert_{L^p}\Vert
h\Vert_{H^q}\leqslant \Vert A\Vert \Vert g\Vert_{H^p}\Vert
h\Vert_{H^q}
\]
and
\begin{eqnarray*}
&&{}\sup\{\vert\langle A\chi_0,\bar{f}\rangle\vert\colon\,f\in
H^1_0, \ \Vert f \Vert_{H^1}\leqslant 1\} \\ \noalign {\smallskip}
&&\makebox[.6cm]{}\leqslant\sup\{\vert\langle
Ag,\bar{h}\rangle\vert\colon\,g\in H^p, \ \Vert
g\Vert_{H^p}\leqslant 1, \ h\in H^q_0, \ \Vert h\Vert_{H^q}\leqslant
1\}\leqslant\Vert A\Vert.
\end{eqnarray*}

This means that $\Phi\colon\,H^1_0\rightarrow{\C}$ defined by
\[
\Phi(f)=\langle
A\chi_0,\bar{f}\rangle=\frac{1}{2\pi}\int_0^{2\pi}(A\chi_0)(e^{i\theta})f(e^{i\theta})\,d\theta
\]
is a bounded linear functional on $H^1_0$. By the Hahn-Banach
theorem this functional can be extended to a bounded functional on
$L^1(\T)$. Hence there exists $\psi\in L^{\infty}(\T)$ such that
\[
\Phi(f)=\frac{1}{2\pi}\int_0^{2\pi}\psi(e^{i\theta})f(e^{i\theta})\,d\theta
\]
for any  $f\in L^1(\T)$.

Therefore if   $f=gh$, where $g\in H^p$ and $h\in H^q_0$, then
\[
\Phi(gh)=\langle A\chi_0,\bar{g}\bar{h}\rangle
=\frac{1}{2\pi}\int_0^{2\pi}\psi(e^{i\theta})g(e^{i\theta})h(e^{i\theta})\,d\theta,
\]
or equivalently by (\ref{chi}),
\begin{equation}
\langle Ag,\bar{h}\rangle=\langle\psi g,\bar{h}\rangle.\label{Ag}
\end{equation}
Taking $h=\chi_n$, $n\geqslant 1$, in (\ref{Ag}) yields
\[
\langle Ag,\bar{\chi}_n\rangle=\langle\psi g,\bar{\chi}_n\rangle
\quad(n\geqslant 1).
\]
Since $Ag\in \overline{H^p_0}$, the last equality implies that  for
$g\in H^p$,
\[
Ag=P_{-}(\psi g)=H_\psi g,
\]
or, in other words, $A=H_\psi$.

\medskip

(ii)\ $\Rightarrow$\ (iii): If $A=H_\psi$ with $\psi\in
L^\infty(\T)$, then for $f\in H^p$,
\begin{eqnarray*}
P_{-}(z\psi f)&=&P_{-}(z(P_{+}(\psi f)+P_{-}(\psi f)))\\
&=&P_{-}(zP_{+}(\psi f))+P_{-}(zP_{-}(\psi f))=P_{-}(zP_{-}(\psi f)).
\end{eqnarray*}
 Thus
\[
P_{-}(zH_\psi f)=P_{-}(zP_{-}(\psi f))=P_{-}(z\psi f)=H_\psi(zf),
\]
or
\[
P_{-}(\mathcal{S}Af)=P_{-}(z\psi f)=H_\psi(zf)=(H_\psi S)f=(AS)f.
\]

\medskip

(iii)\ $\Rightarrow$\ (i): Assume now that a bounded operator
$A\colon\,H^p\rightarrow \overline{H^p_0}$ satisfies
$P_{-}(\mathcal{S}A)=AS$. Then for $j\geqslant 1$ and $k\geqslant
1$,
\begin{eqnarray*}
\langle A\chi_k,\bar{\chi}_j\rangle&=&\langle
AS\chi_{k-1},\bar{\chi}_j\rangle=
\langle P_{-}(\mathcal{S}A)\chi_{k-1},\overline{\chi}_j\rangle \\
&=& \langle \mathcal{S}A\chi_{k-1},\bar{\chi}_j\rangle=\langle
A\chi_{k-1},\bar{\chi}_{j+1}\rangle,
\end{eqnarray*}
since $\langle P_{+}(\mathcal{S}A)\chi_{k-1},\bar{\chi}_j\rangle=0$.
This implies that the operator $A$ is represented by the Hankel
matrix with respect to the basis $\{\chi_k\}_{k\geqslant 0}$ in
$H^p$ and $\{\bar{\chi}_j\}_{j\geqslant 1}$ in $\overline{H^p_0}$.

\end{proof}

Let $S^I$ be the compression of the shift $S$ on $H^p$ to the
subspace $K^p_I$, i.e.
\[
S^If= P_{I}zf,
\]
for  $f\in K_I^p$. Also let for $\varphi\in H^{\infty}$ the operator
$T^I_{\varphi}f= P_I(T_{\varphi}f)$ be the compression of the
Toeplitz operator $T_{\varphi}$ on $H^p$ to $K_I^p$.

Now our aim is to prove the following
\begin{thm} If $T$ is a bounded linear operator on
$K_I^p$, $1<p<\infty,$ that commutes with the operator  $S^I$, then
there exists $\varphi\in H^{\infty}$ such that $T= T_{\varphi}^I$.
\end{thm}

We first prove the following lemma (cf. \cite[p. 181]{N1},
\cite[Lemma 2.2, p. 14]{P}).
\begin{lem}  Let $T$ be a bounded linear operator on $K_I^p$, $1<p<\infty,$ and let
an operator $A\colon\,H^p\rightarrow \overline{H^p_0}$ be given by
\[
 Af= \bar{I}TP_{I}f.
 \]
Then $T$ commutes with $S^I$ if and only if $A$ is a Hankel
operator.
\end{lem}
\begin{proof}
In view of Theorem 4.1, $A$ is a Hankel operator if and only if
\begin{equation}
P_{-}(zAf)= Azf,\quad f\in H^p,\label{H1}
\end{equation}
thus
\begin{equation}
P_{-}(z\bar{I}TP_{I}f)=\bar ITP_{I}zf, \quad f\in H^p,\label{H2}
\end{equation}
or equivalently
\begin{equation}
IP_{-}(z\bar{I}TP_{I}f)=TP_{I}zf, \quad f\in H^p.\label{H3}
\end{equation}
By formula (\ref{projection}) we have
\[
IP_{-}(z\bar{I}TP_{I}f)= P_IzTP_{I}f=S^ITP_If.
\]
Now observe that for  $f\in IH^p$ the left-hand side of the last
equality equals to zero since $P_If=0$. Hence  it follows from
(\ref{H3}) that $TP_{I}zf=0$ for $f\in IH^p$. Consequently,
\[ I
P_{-}(z\bar{I}TP_{I}f)=TS^If
\]
for  $f\in K^p_I$.

Finally, note that for $f\in K^p_I$,
\[
IP_{-}(z\bar{I}TP_{I}f)=IP_{-}z\bar{I}Tf=P_IzTf=S^ITf
\]
since $P_If=f$. Therefore equality  (\ref{H1}) (or (\ref{H2})) is
equivalent to
\[
S^ITf=TS^If,\quad f\in K^p_I.
\]
\end{proof}

\noindent {\em Proof of Theorem 4.2.}  \ If an operator $T$
satisfies assumptions of Theorem 4.2, then the operator $A$ from
Lemma 4.3 is a bounded Hankel operator. In view of Theorem 4.1 there
exists $\psi\in L^{\infty} $ such that $Af=H_{\psi}$. Thus \[
 Af=\bar{I}TP_{I}f= H_{\psi}f=P_{-}\psi f
\]
for $f\in H^p$. As in the proof of Lemma 4.3, this equality implies
$P_{-}\psi f=0$ for $f\in IH^p$, and so
\[
P_{-}\psi I
f=H_{\psi I}f=0\quad \text{for}\ \ f\in H^p.
\]
Put $\varphi=\psi I$. Then $\varphi\in H^{\infty}$, $\psi=\varphi
\bar{I}$ and for $f\in K^p_I$ (i. e. for $f$ such that $P_If=f$)
\[
\bar{I}Tf= P_{-}\bar{I}\varphi f,
\]
and so
\[
 Tf=I P_{-}\bar{I}\varphi f=P_I\varphi f= T_{\varphi}^If.
\]
{}\hfill$\square$

\medskip

We know that under pairing (\ref{pairing}), the dual space of
$K^p_I$ can be identified with $K^q_I$ (see e.g. \cite[p. 109]{CR}).
Moreover, the adjoint of the compression of the shift operator $S$
to $K_I^p$, i.e. $\left(P_IS|K_I^p\right)^\ast$ is equal to
$T_{\bar{z}}|K_I^q$. Similarly, for $\varphi\in H^\infty$, the
adjoint of the compression of $T_\varphi$ to $K^p_I$, i.e.
$\left(P_IT_\varphi|K_I^p\right)^\ast$ is equal to
$T_{\bar{\varphi}}|K_I^q$.

Theorem 4.2 implies the following

\begin{cor}
For $1<p<\infty$,
\[
\{S^\ast| K^p_I\}'=\{T_{\bar{\varphi}}\,| K^p_I\colon\,\varphi\in
H^\infty\}.
\]
\end{cor}

\vspace{.2in}

{\bf Final Remark.}  After this paper was submitted for publication
we learned from A. Hartmann that Theorem 4.2 had been proved in his
paper \cite[Th\'eor\`eme 1.14]{H1}. He observed that the
$H^\infty$-functional calculus of B.\,Sz.-Nagy and Foia\c{s} for the
compressed shift operator to $K^p_I$, $1<p<\infty$, can be defined
exactly in the same way as in the case $p=2$. Therefore, using the
same reasoning, Sarason's theorem on the commutant of the compressed
shift can be extended to the non-Hilbertian case.

Furthermore, the proof of \cite[Lemma 4.9, p. 1270--1271]{FHR} can be carried out in the case $1<p<\infty$ by means of this generalized $H^\infty$-functional calculus for the compressed shift operator.

The advantage of our  presentation is that the proofs of Theorems
3.2 and 4.2 are direct. We use neither the functional calculus, nor
the commutant lifting theorem.

\vspace{.2in}

\noindent
{\bf Acknowledgement.} The second named author would like to thank the Institute of Mathematics of the Maria Curie-Sk{\l}odowska University for supporting his visit to Lublin where part of this paper was written.


\vspace{.4in}

\begin{thebibliography}{99999}

\bibitem[BS]{BS}
A. B\"{o}ttcher and B. Silbermann, {\em Analysis of Toeplitz
Operators\/}, Springer Monographs in Mathematics, Springer-Verlag,
Berlin 1990.
\bibitem[CR]{CR}
J.\,A. Cima and W.\,T. Ross, {\em The Backward Shift on the Hardy
Space\/}, Mathematical Surveys and Monographs, Vol. 79, American
Mathematical Society, Providence, RI 2000.
\bibitem[Dy1]{Dy1} K.\,M. Dyakonov, {\em Factorization and non-factorization theorems for pseudocontinuable functions\/},
 Adv. Math. {\bf 320} (2017), 630–-651.
\bibitem[Dy2]{Dy2} K.\,M. Dyakonov, {\em A free interpolation problem for a subspace of $H^{\infty}$ \/}, Bull. Lond. Math. Soc. {\bf 50} (2018), no. 3, 477--486.
\bibitem[FHR]{FHR} E. Fricain, A. Hartmann, and W.\,T.  Ross, {\em Range spaces of co-analytic Toeplitz operators\/}, Canad. J. Math. {\bf 70} (2018), no. 6, 1261--1283.
\bibitem[F1]{F1}  P.\,A. Fuhrmann, {\em On the corona theorem and its application to spectral problems in Hilbert space \/}, Trans. Amer. Math. Soc. {\bf 132} (1968), 55--66.
\bibitem[F2]{F2}  P.\,A. Fuhrmann, {\em  Exact controllability and observability and realization theory in Hilbert space \/}, J. Math. Anal. Appl. {\bf 53} (1976), no. 2, 377--392.
\bibitem[GMR]{GMR} S.\,R. Garcia, J. Mashreghi, W.\,T. Ross, {\em
Introduction to model spaces and their operators\/},  Cambridge
University Press, Cambridge, 2016.
\bibitem[H1]{H1}
A. Hartmann, {\em Une approche de l'interpolation libre g\'en\'eralis\'ee par la th\'eorie des op\'erateurs et caract\'erisation des traces $H^p|_{\Lambda}$}, J. Operator Theory {\bf 35} (1996), no. 2, 281--316.
\bibitem[H2]{H2}
A. Hartmann, {\em The generalized Carleson condition in certain spaces of analytic functions\/}, pp. 245--260; in: Proceedings of the 13th International Conference on Banach Algebras held at the Heinrich Fabri Institute of the University of T\"{u}bingen in Blaubeuren, July 20--August 3, 1997, red. E. Albrecht, M. Mathieu, De Gruyter Proceedings in Mathematics, Berlin 1998.
\bibitem[HS]{HS}
A. Hartmann and K. Seip, {\em Extremal functions as divisors for
kernels of Toeplitz operators\/}, J. Funct. Anal. {\bf 202} (2003),
no. 2, 342–-362.
\bibitem[M]{M}
V. M\"{u}ller, {\em  Spectral theory of linear operators and
spectral systems in Banach algebras\/}, Operator Theory: Advances
and Applications, 139. Birkhäuser Verlag, Basel, 2007.
\bibitem[N1]{N1}
N.\,K. Nikolskii, {\em Treatise on the shift operator\/},
Springer-Verlag, New York 1986.
\bibitem[N2]{N2} N.\,K. Nikolskii, {\em Operators,
Functions, and Systems: An Easy Reading. Vol. I: Hardy, Hankel, and
Toeplitz\/}, Mathematical Surveys and Monographs, Vol. 92, American
Mathematical Society, Providence, RI 2002.
\bibitem[N3]{N3}
N.\,K. Nikolskii, {\em Operators, functions, and systems: an easy
reading. Vol. 2, Model operators and systems\/},  American
Mathematical Society, Providence, RI, 2002.
\bibitem[P]{P}
V.\,V. Peller, {\em Hankel Operators and Their Applications\/},
Springer Monographs in Mathematics, Springer-Verlag, New York 2003.
\bibitem[R]{R}
W. Rudin, {\em Functional Analysis\/}, McGraw Hill, Inc., New York
1973.
\bibitem[S1]{S1}
D. Sarason, {\em Generalized interpolation in $H^\infty$\/}, Trans.
Amer. Math. Soc. {\bf 127} (1967), 179--203.
\bibitem[St]{St}
D.\,A. Stegenga, {\em Bounded Toeplitz Operators on $H^1$ and
Applications of the Duality Between $H^1$ and the Functions of
Bounded Mean Oscillation\/}, Amer. J. Math. {\bf 98}, no. 3 (1976),
573--589.
\bibitem[SzNB]{SzNB}
B. Sz.-Nagy and  C. Foia\c{s}, {\em Harmonic analysis of operators on
Hilbert space\/},  North-Holland, Amsterdam 1970.
\bibitem[Z]{Z} K.Zhu, {\em Spaces of holomorphic functions in the unit ball\/},  Springer-Verlag, New York 2005.
\end{thebibliography}
\end{document}